\documentclass[reqno]{amsart}
\usepackage[utf8]{inputenc}
\usepackage{amssymb}
\usepackage{mdframed}
\usepackage{bm,bbm}
\usepackage{euscript}
\usepackage{xcolor}
\usepackage{tikz-cd}
\usepackage{comment}
\usepackage[all,cmtip]{xy}

\usepackage[left= 3.4 cm, right= 3.4 cm, top= 2.8 cm, bottom=2.7 cm, foot= 1.2 cm, marginparwidth=2.7 cm, marginparsep=0.5 cm]{geometry}

\numberwithin{equation}{section}
\usepackage{soul}

\usepackage{hyperref}
\hypersetup{colorlinks}
\definecolor{darkred}{rgb}{0.5,0,0}
\definecolor{darkgreen}{rgb}{0,0.5,0}
\definecolor{darkblue}{rgb}{0,0,0.5}
\hypersetup{colorlinks, linkcolor=darkblue, filecolor=darkgreen, urlcolor=darkred, citecolor=darkblue}
\makeatletter 
\@addtoreset{equation}{section}
\makeatother  

\numberwithin{equation}{section}

\newtheorem{thm}{Theorem}[section]
\newtheorem{cor}[thm]{Corollary}

\newtheorem{prop}[thm]{Proposition}

\newtheorem{lemma}[thm]{Lemma}
\newtheorem{def-lemma}[thm]{Definition-Lemma}
\theoremstyle{definition}
\newtheorem{defn}[thm]{Definition}
\theoremstyle{remark}
\newtheorem{rem}[thm]{Remark}

\newtheorem{example}[thm]{Example}

\newcommand{\beq}{\begin{equation}}
\newcommand{\eeq}{\end{equation}}
\newcommand{\beqn}{\begin{equation*}}
\newcommand{\eeqn}{\end{equation*}}
\newcommand{\ov}{\overline}
\newcommand{\mb}{\mathbb}

\newcommand{\scrC}{\EuScript C}

\newcommand{\scrM}{\EuScript M}

\newcommand{\scrS}{\EuScript S}

\makeatletter
\newcommand{\colim@}[2]{%
  \vtop{\m@th\ialign{##\cr
    \hfil$#1\operator@font colim$\hfil\cr
    \noalign{\nointerlineskip\kern1.5\ex@}#2\cr
    \noalign{\nointerlineskip\kern-\ex@}\cr}}%
}
\newcommand{\colim}{%
  \mathop{\mathpalette\colim@{\rightarrowfill@\textstyle}}\nmlimits@
}
\makeatother

\title[Equivariant formality]{Equivariant formality in complex-oriented theories}

\author{Shaoyun Bai}
\address{Department of Mathematics, Columbia University, New York, NY 10027, USA}
\email{sb4841@columbia.edu}

\author{Daniel Pomerleano}
\address{University of Massachusetts, Boston, 100 William T, Morrissey Blvd, Boston, MA 02125, USA}
\email{Daniel.Pomerleano@umb.edu}

\begin{document}

\begin{abstract}
Let $G$ be a product of unitary groups and let $(M,\omega)$ be a compact symplectic manifold with Hamiltonian $G$-action. We prove an equivariant formality result for any complex-oriented cohomology theory $\mathbb{E}^*$ (in particular, integral cohomology). This generalizes  the celebrated result of Atiyah--Bott--Kirwan for rational cohomology from the 1980s. The proof does not use classical ideas but instead relies on a recent cohomological splitting result of Abouzaid--McLean--Smith \cite{AMS} for Hamiltonian fibrations over $\mathbb{CP}^1.$ Moreover, we establish analogues of the ``localization" and ``injectivity to fixed points" theorems for certain cohomology theories studied by Hopkins--Kuhn--Ravenel in \cite{HKR}. As an application of these results, we establish a Goresky--Kottwitz--MacPherson theorem with Morava $K$-theory coefficients for Hamiltonian $T$-manifolds.
\end{abstract}

\maketitle

\section{Introduction}

In \cite{GKM}, Goresky, Kottwitz, and MacPherson introduced the notion of an equivariantly formal space. Let $G$ be a compact, connected Lie group, $T \subset G$ a maximal torus and $R$ a commutative ring. A $G$-space $M$ is called equivariantly formal over $R$ if the Leray spectral sequence
\begin{align} H^p(BG;H^q(M;R))=> H_G^{p+q}(M;R) \end{align}
degenerates. The main examples of equivariantly formal-spaces (c.f. \cite[\S 14.1]{GKM})  arise from the following celebrated result of Atiyah--Bott--Kirwan (\cite[\S 7]{Atiyah-Bott}, \cite[\S 5]{Kirwan}):

\begin{thm}\label{thm:AB} Let $(M,\omega)$ be a compact symplectic manifold equipped with a Hamiltonian $G$-action. Then $M$ is equivariantly formal over any $\mathbb{Q}$-algebra $R$. \end{thm}

The proof of Theorem \ref{thm:AB} easily reduces to the case $G=T$ via taking Weyl invariants. In the torus-equivariant case, the idea is to argue that the moment map for a generic one-parameter subgroup of $T$ is equivariantly perfect. Theorem \ref{thm:AB} naturally leads one to ask what happens with integral coefficients. The proof of this claim in these classical references relies on the so-called Atiyah--Bott lemma (\cite[Proposition 13.4]{AtiyahBottYM},\cite[Lemma 2.18]{Kirwan}), whose proof may fail when $R=\mathbb{Z}$ (or even $R=\mathbb{F}_p)$ and the cohomology of the critical sets have torsion (see e.g. \cite[\S 13]{AtiyahBottYM}, \cite[Remark 2.4]{Harada}, \cite[\S 4]{Tolman-Weitsman}). The main result of this note is the following somewhat surprising generalization of Theorem \ref{thm:AB} to any complex-oriented cohomology theory $\mathbb{E}^*$, including integral cohomology as a special case:

\begin{thm} \label{thm:general} Let $(M,\omega)$ be a compact symplectic manifold with Hamiltonian $G$-action. \begin{enumerate} \item Suppose $G$ is a product of unitary groups. For any complex oriented cohomology theory $\mathbb{E}^*$, the Leray spectral sequence \begin{align}\label{eq:LeraySerreE} H^p(BG;\mathbb{E}^q(M))=> \mathbb{E}_G^{p+q}(M) \end{align} degenerates. Moreover, there is a splitting:
 \begin{align} \label{eq:strongsplitting} \mathbb{E}^*_G(M) \cong \mathbb{E}^*(M \times BG) \end{align} of $\mathbb{E}^*(BG)$ modules.\footnote{Contrary to the usual convention in equivariant homotopy theory, we use the notation $\mathbb{E}_G^*(M)$ to denote the Borel-equivariant cohomology.}  \item For general $G$, suppose that $\mathbb{E}^*(pt)$ is an algebra over a field $\mathbf{k}$ for which the natural map  \begin{align} \label{eq:GT} H^*(BG;\mathbf{k}) \to H^*(BT;\mathbf{k}) \end{align} is injective. Then the Leray spectral sequence \eqref{eq:LeraySerreE} degenerates.
 \end{enumerate}
\end{thm}

Specializing to the case of classical cohomology implies the following result, which already appears to be new:

\begin{thm}\label{thm:cohomology} Let $(M,\omega)$ be a compact symplectic manifold with Hamiltonian $G$-action. \begin{enumerate} \item If $G$ is a product of unitary groups, then $M$ is equivariantly formal over any commutative ring $R$. In fact, there is a (non-canonical) isomorphism of $H^*(BG;R)$-modules \begin{align} \label{eq:splittingintrocohR} H_G^*(M;R) \cong H^*(M;R)\otimes_R H_G^*(pt;R). \end{align} \item For general $G$, let $\mathbf{k}$ be a field such that the natural map  \begin{align} H^*(BG;\mathbf{k}) \to H^*(BT;\mathbf{k}) \end{align} is injective and let $R$ be any $\mathbf{k}$-algebra. Then $M$ is equivariantly formal over $R.$ \end{enumerate}  \end{thm}

\begin{rem}
As far as the authors know, for $R = \mathbb{Z}$, the sharpest known results in the literature along the lines of Theorem \ref{thm:cohomology} (1) for a general Hamiltonian $G$-manifold $M$ were: 
\begin{itemize}
\item the short exact sequence (not proven to be split) for cohomology groups in degree $2$ (cf. \cite[Theorem A.3]{Harada})
\begin{equation}
    0 \longrightarrow H^2_{G}(pt;\mathbb{Z}) \otimes_{\mathbb{Z}} H^0(M;\mathbb{Z}) \longrightarrow H^2_G(M;\mathbb{Z}) \longrightarrow H^2(M;\mathbb{Z}) \longrightarrow 0.
\end{equation}
\item for $G$ a torus, equivariant formality over $\mathbb{Z}$ was known when all the isotropy groups are \emph{connected} (cf. \cite[Theorem 5.1]{franz-puppe}).
\end{itemize}
\end{rem}

\begin{rem}
 When $R$ is not a field, the splitting from \eqref{eq:splittingintrocohR} does not follow from equivariant formality; see \cite[Example 5.2]{franz-puppe} for a simple example of a space which is equivariantly formal but for which no splitting of this form exists. \end{rem} 
 \begin{rem} For semi-simple Lie groups, Theorem \ref{thm:cohomology} (2) is sharp in the following sense. The cohomology $H^*(BT,\mathbf{k})$ is isomorphic to $H_G^*(G/T,\mathbf{k})$, where $G$ acts by left-translation. This action is Hamiltonian for an appropriate symplectic form on the flag manifold $G/T$, and so if \eqref{eq:GT} is not injective, $G/T$ is not equivariantly formal over $\mathbf{k}.$   \end{rem}

\begin{example} We illustrate Theorem \ref{thm:cohomology} (2) for some classical groups:  \begin{itemize} \item If $G=U(n),SU(n),$ or $Sp(2n)$, then the result holds over any field $\mathbf{k}$; see e.g. \cite[Theorem 2.1]{toda1987} for the calculation of $H^*(BG;\mathbf{k})$ for classical groups $G$.  \item If $G=SO(n)$, then the result holds for any field $\mathbf{k}$ with $\operatorname{char}(\mathbf{k}) \neq 2.$ \item If $G=G_2$, then the result holds provided $\operatorname{char}(\mathbf{k}) \neq 2$ (\cite[Theorem 2.19]{Akbulut}). \end{itemize} \end{example}

Theorem \ref{thm:cohomology} may be unexpected as many of the other well-known consequences of the Atiyah--Bott lemma are known to fail with integral coefficients (or finite field coefficients). For example, it is easy to construct examples of Hamiltonian $T$-manifolds $M$ for which the restriction to the fixed point locus $j^*: H_T^*(M,\mathbb{Z}) \to H_T^*(M^{T},\mathbb{Z})$ fails to be injective. Indeed, we do not know how to prove Theorem \ref{thm:general} (or even Theorem \ref{thm:cohomology}) using this circle of ideas. 
 
 Instead, we deduce Theorem \ref{thm:general} from a more general result in symplectic topology. Recall that a locally trivial fiber bundle $M \to P \to B$ with fiber a closed symplectic manifold $(M,\omega)$ is said to be Hamiltonian if its structure group can be reduced to the group of Hamiltonian symplectomorphisms $\operatorname{Ham}(M,\omega)$. We prove the following cohomological splitting result for Hamiltonian fibrations:

 \begin{prop} \label{prop:generalized} Let $M \to P \to B$ be Hamiltonian fibration over a base $B$ which is a product of complex Grassmannians. Then there is a non-canonical isomorphism  of $\mathbb{E}^*(B)$ modules \begin{align} \mathbb{E}^*(P) \cong \mathbb{E}^*(M)\otimes_{\mathbb{E}^{*}(pt)} \mathbb{E}^*(B) \end{align} \end{prop}

The case where $B=\mathbb{CP}^1$ is an important recent result of Abouzaid--McLean--Smith \cite{AMS} (see also \cite{Bai-Xu} for a more direct argument in the case of integral cohomology). Their argument is based on ``hard" pseudo-holomorphic curve methods first introduced by \cite{LMP}. We are able to formally ``bootstrap" this result to prove Proposition \ref{prop:generalized} by sharpening the ``soft" topological approach to cohomological splitting of Hamiltonian fibrations by Lalonde--McDuff \cite{McDuff-Lalonde}.  

\begin{rem} In the introduction to their paper \cite{AMS}, the authors suggest that one could use the arguments of \cite{McDuff-Lalonde} to prove Proposition \ref{prop:generalized} for $B$ a product of projective spaces and cohomology with coefficients in any \emph{field}. The goal of this note is to carry out and considerably generalize this idea (by establishing cohomological splitting for Hamiltonian fibrations over more general $B$ and with coefficients in any complex-oriented cohomology theory) as well as to make the connection to classical concepts in equivariant symplectic geometry.  \end{rem}

To finish this note, we include two applications of our equivariant formality result. Our applications both concern a class of cohomology theories studied by Hopkins--Kuhn--Ravenel in \cite{HKR} (cf. Definition \ref{defn:HKR}). For the first application, we  establish an Atiyah--Bott style localization theorem (Lemma \ref{lem:localization}) for HKR theories which are futhermore integral and of finite Krull dimension in the sense of Definition \ref{defn:HRKfinite}. Combining this localization theorem with our equivariant formality result, we deduce the following:

\begin{thm}\label{thm:highlight}
Let $(M,\omega)$ be a compact symplectic manifold with Hamiltonian $T$-action. Let $\mathbb{E}^*$ be an HKR cohomology theory which is integral and of finite Krull dimension. Suppose that $\mathbb{E}^*(M)$ is free as a module over the ground ring $\mathbb{E}_*$. Then the restriction map on equivariant cohomology \begin{align} j^*:\mathbb{E}_T^*(M) \to \mathbb{E}_T^*(M^{T})\end{align} is injective. In particular, if $\mathbb{E}_*$ is a graded field, $j^*$ is injective for any compact symplectic manifold $M$ with Hamiltonian $T$-action.\end{thm}

This is in contrast to the aforementioned failure of injectivity of $j^*$ for ordinary cohomology with finite field coefficients. The ingredients in the proof can further be used to prove the Atiyah--Bott ``integration formula" with $\mathbb{E}^*$-coefficients, which is completely analogous to the classical statement for the rational cohomology; see Theorem \ref{AB-localization}.

The second application concerns one of the most interesting classes of HKR theories, the Morava $K$-theories $K_{p}(n)$. For $\mathbb{E} = K_{p}(n)$, we can prove the following Goresky--Kottwitz--MacPherson type theorem which calculates the (equivariant) Morava $K$-theories of GKM spaces. 

Let $(M,\omega)$ be a compact symplectic manifold with Hamiltonian torus action by $T = (S^1)^m$. It is said to satisfy the \emph{GKM conditions} if $T$ acts with only finitely many fixed points $x_1, \dots, x_k$, and the union of $T$-orbits of corank at most $1$ is a union of invariant two-spheres $\ov{E}_1, \dots, \ov{E}_l$. Each $\ov{E}_j$ is isomorphic to $S^2 = \mathbb{C}^{\times} \cup \{x_{j_0}\} \cup \{x_{j_{\infty}}\} = E_{j} \cup \{x_{j_0}\} \cup \{x_{j_{\infty}}\}$ for some $x_{j_0}, x_{j_{\infty}} \in \{x_1, \dots, x_k\}$, and the action of $T$ on $E_j$ factors through a group homomorphism $\Theta_j: T \to S^1$ where $S^1$ rotates the sphere $\ov{E}_j$ with fixed points $x_{j_0}$ and $x_{j_{\infty}}$. 

\begin{thm}\label{thm:GKM}
 Let $(M,\omega)$ be a compact symplectic manifold with Hamiltonian torus action satisfying the GKM conditions and let $\mathbb{E}^*=K_p(n)$. The restriction to fixed points
\begin{equation}
    \mathbb{E}^*_{T}(M) \to \bigoplus_{x_i, 1 \leq i \leq k} \mathbb{E}^{*}(pt)[\![u_1, \dots, u_m]\!]
\end{equation}
is injective. Furthermore, the image is given by
\begin{equation}
    \big\{ (f_1, \dots, f_k) \quad | \quad \forall 1 \leq j \leq l, \quad f_{j_0} |_{\mathbb{E}^*_{\ker(\Theta_j)}} = f_{j_{\infty}} |_{\mathbb{E}^*_{\ker(\Theta_j)}}\big\},
\end{equation}
where the restriction $|_{\mathbb{E}^*_{\ker(\Theta_j)}}$ denotes the homomorphism $\mathbb{E}^*_{T}(pt) \to \mathbb{E}^*_{\ker(\Theta_j)}(pt)$ induced by the inclusion $\ker(\Theta_j) \hookrightarrow T$ of the kernel of $\Theta_j$.
\end{thm}

This statement reduces the calculation of the cohomology rings $\mathbb{E}^*(M)$ and $\mathbb{E}^*_{T}(M)$ to the combinatorics of the $T$-invariant $2$-spheres. We leave the applications of this circle of ideas to topology and enumerative geometry for future investigations.

\subsection*{Acknowledgements}
We thank David Anderson, Yael Karshon, Allen Knuston, and Che Shen for sharing their expertise on Bott--Samelson resolutions. D.P. was partially funded by NSF grant DMS-2306204 while working on this note.

\section{Complex-oriented splitting}
\subsection{Complex-oriented theories}

Suppose $\mathbb{E}^*$ is a complex-oriented cohomology theory with coefficient ring $\mathbb{E}_*$. Recall that this means that there exists a class $u \in \mathbb{E}^2(BS^1)$ that
restricts to a generator of the reduced cohomology,  $\tilde{\mathbb{E}}^*(\mathbb{CP}^1)$ (a free rank one $\mathbb{E}_*$–module). Such
a class $u$ is called a complex orientation of $\mathbb{E}^*$. The universal example of such a cohomology theory is complex cobordism $MU^*$. For a multiplicative cohomology theory $\mathbb{E}^*$, complex orientations are in bijective correspondence with multiplicative natural transformations to $\mathbb{E}^*$ from complex cobordism $MU^*$. A choice of orientation $u \in \mathbb{E}^2(BS^1)$ gives rise to isomorphisms \begin{align} \mathbb{E}^*(BS^1)=\mathbb{E}_*[\![u]\!], \quad \mathbb{E}^*(B(S^1)^m)=\mathbb{E}_*[\![u_1,u_2,\cdots,u_m]\!].  \end{align}

It follows from this that the complex orientation  of $\mathbb{E}^*$ has an associated formal group law. Let $\mu: BS^1\times BS^1 \to BS^1$ denote the universal classifying map for tensor product of line bundles.  The formal group law is by definition the pull-back of $u$ under this classifying map: \begin{align} u_1+_{L_{\mathbb{E}}} u_2= L_\mathbb{E}(u_1,u_2)=\mu^*(u) \in \mathbb{E}_*[\![u_1,u_2]\!]. \end{align} From the formal group law, one can extract the $\ell$-series, which is given by: \begin{align}[\ell]\cdot_{\mathbb{E}}u =\underbrace{u+_{L_{\mathbb{E}}} u+_{L_{\mathbb{E}}} \cdots +_{L_{\mathbb{E}}} u}_{l \text{ times}} \in \mathbb{E}_*[\![u]\!].\end{align}

We also record the following technical fact which will be used in the sequel: 

\begin{lemma} \label{lem:AtiyahHirzebruch} Let $\mathbb{E}^*$ be a complex-oriented cohomology theory and let $B$ be a finite-type CW complex (meaning $B$ has finitely many cells in each degree) with only even cells. The Atiyah--Hirzebruch spectral sequence\begin{align} H^p(B,\mathbb{E}^q(pt))=> \mathbb{E}^{p+q}(B) \end{align} degenerates. \end{lemma}
\begin{proof} Because the cohomology groups $H^p(B,\mathbb{Z})$ are free, the $E_2$ page of the spectral sequence can be identified with \begin{align} E_2^{p,q}= H^p(B)\otimes \mathbb{E}^q(pt). \end{align} The spectral sequence is linear over $\mathbb{E}^q(pt)$ and it suffices to prove that \begin{align} d^r(\alpha\otimes 1)=0 \text{ for all } \alpha \in H^*(B). \end{align} If $\mathbb{E}^*=MU^*$ is complex cobordism the spectral sequence collapses because the $E_2$ page is concentrated in even degree. In general, the complex orientation of $\mathbb{E}^*$ is induced by a natural transformation (of multiplicative cohomology theories) $MU^* => \mathbb{E}^*$ and the Atiyah-Hirzebruch sequence is natural in the cohomology theory. The result therefore follows from the case of complex cobordism.  \end{proof}

\subsection{Splittings}
Next, we discuss cohomological splittings with $\mathbb{E}^*$ coefficients, where $\mathbb{E}^*$ is a complex-oriented cohomology theory.

\begin{defn} Suppose that we have a fibration $M \xrightarrow{i} P \xrightarrow{\pi} B$. A fibration is \emph{c-split} if there exists a map of $\mathbb{E}_*$ modules $s:\mathbb{E}^*(M) \to \mathbb{E}^*(P)$ so that $i^* \circ s=\operatorname{id}.$ We will refer to $s$ as a \emph{c-splitting}. \end{defn}

\begin{defn} A smooth manifold $B$ is called \emph{Hamiltonian c-split} for $\mathbb{E}^*$ if every Hamiltonian fibration $M \xrightarrow{i} P \xrightarrow{\pi} B$ is c-split. \end{defn}

We suppress the cohomology theory $\mathbb{E}$ when it is clear from the context. The key result for us is the following result of Abouzaid--McLean--Smith: 

\begin{thm}[\cite{AMS}] \label{thm:AMS} The one-dimensional projective space, $\mathbb{CP}^1$, is Hamiltonian c-split. \end{thm}

\begin{rem} When $\mathbb{E}^*$ is ordinary cohomology with coefficients in a commutative ring $R$, the first author and Xu (\cite{Bai-Xu}) have given a more direct proof of Theorem \ref{thm:AMS} which avoids the chromatic homotopy theory arguments from \cite{AMS}. \end{rem}

The following two Lemmas are refinements of arguments in \cite{McDuff-Lalonde} (cf. Corollary 4.2 and Lemma 4.3 of \emph{loc. cit.}).

 \begin{lemma} \label{lem:blowups1}
 Let $\tau^{B}:\tilde{B} \to B$ be a  birational map of smooth projective varieties. If  $\tilde{B}$ is Hamiltonian c-split, then so is $B$. 
 \end{lemma}
 \begin{proof} Suppose that we have a Hamiltonian fibration $M \xrightarrow{i} P \xrightarrow{\pi} B$ over $B$ and let $M \xrightarrow{\tilde{i}} \tilde{P} \xrightarrow{\tilde{\pi}} \tilde{B}$ be the pull-back Hamiltonian fibration, which admits a map  $\tau^P:\tilde{P} \to P.$ We can factor $\tau^B$ as an inclusion followed by a projection $\tilde{B} \to \mathbb{CP}^r \times B \to B$ and similarly for $\tau^P$. This shows that the map $\tau^P$ is compatible with complex-orientations. Set $U \subset B$ to be the locus over which $\tau^{B}$ is an isomorphism and let $P_U$ denote the restriction of the bundle to this locus. Note that we can assume that our reference fiber $M_b$ for both $\tilde{P}$ and $P$ lies over a point $b \in U$. We have a Cartesian diagram of complex-oriented manifolds/maps: \[
\xymatrix{
 P_U \ar[d] \ar[r]^{\tilde{j}} & \tilde{P} \ar[d]^{\tau^P}  \\
    P_U \ar[r]^{j} & P.} \] 
  For such Cartesian squares, the Umk\"{e}hr maps on cohomology are compatible with pullback/restriction maps; see \cite[1.7]{Quillen} for complex cobordism, \cite[E(iii)]{Bressler} in general, and \cite[Proposition 12.9]{MR1646248} for a statement at the level of spectra, which implies the statement for cohomology by taking  homotopy groups. In the present case this means that: \begin{align} \label{eq:gysincommutes} j^* \circ \tau^P_*=\tilde{j^*} \text{ and hence } i^* \circ \tau^P_* =\tilde{i}^*. \end{align}
  Suppose we are given a splitting $\tilde{s}:\mathbb{E}^*(M) \to \mathbb{E}^*(\tilde{P})$ such that $(\tilde{i})^* \circ \tilde{s}= \operatorname{id}$. Then it follows from  \eqref{eq:gysincommutes} that \begin{align} s=\tau^P_* \circ \tilde{s}:\mathbb{E}^*(M) \to \mathbb{E}^*(P) \end{align} satisfies $i^* \circ s=\operatorname{id}.$
    \end{proof}

\begin{lemma} \label{lem:tower1} Suppose that we have a Hamiltonian fibration $M \xrightarrow{i} P \xrightarrow{\pi} B$ with $B$ simply connected. Suppose that $M$ and $B$ are Hamiltonian c-split. Then $P$ is also Hamiltonian c-split.  \end{lemma}
\begin{proof} Let $F \to E \to P$ be a Hamiltonian bundle over $P$. After restricting to $M$ we get a Hamiltonian bundle $F \to W \to M$ which by assumption gives rise to a c-splitting:  \begin{align}\label{eqn:proj-fibration} s_M: \mathbb{E}^*(F) \to \mathbb{E}^*(W). \end{align} The composite bundle $ W \to E \to B$ is also Hamiltonian \cite[Lemma 3.9]{McDuff-Lalonde} and thus gives rise to a c-splitting:  \begin{align} s_B: \mathbb{E}^*(W) \to\mathbb{E}^*(E). \end{align} The desired c-splitting for $F \to E \to P$ is the map $s=s_B \circ s_M$. \end{proof}

\begin{lemma} \label{lem:cohsplitproj} Let $B$ be a product of complex projective spaces. Then $B$ is Hamiltonian c-split.  \end{lemma}
\begin{proof} In view of Lemma \ref{lem:tower1}, it suffices to consider the case where $B=\mathbb{CP}^{k}$ is a single projective space. We proceed by induction on $k$. The base case where $k=1$ is Theorem \ref{thm:AMS}.  Consider the blow-up $\tilde{B}$ of $\mathbb{CP}^{k}$ at a point. Then $\tilde{B}$ can be identified with the $\mathbb{CP}^1$-bundle: \begin{align} \mathbb{P}(\mathcal{O}_{\mathbb{CP}^{k-1}}\oplus \mathcal{O}_{\mathbb{CP}^{k-1}}(-1)) \to \mathbb{CP}^{k-1}. \end{align}  The inductive hypothesis together with Lemma \ref{lem:tower1} shows that $\tilde{B}$ is Hamiltonian c-split. It then follows from Lemma \ref{lem:blowups1} that $\mathbb{CP}^{k}$ is c-splitting.  \end{proof}

More generally, we have the following result:

\begin{lemma} \label{lem:cohsplit} Let $B$ be a product of Grassmannians. Then $B$ is Hamiltonian c-split.  \end{lemma}
\begin{proof}
Again, it suffices to consider a single Grassmannian. 
Recall that the Schubert varieties of the complete flag variety $GL(n,{\mb C})/B$ are indexed by elements $w$ of the symmetric group $S_n$. If we consider the Schubert variety $Z_w$ corresponding to the permutation $w(i) = i+k (\operatorname{mod} n)$, this maps birationally onto the complex Grassmannian $\mathrm{Gr}(k,n)$ via the projection $GL(n,{\mb C})/B \to \mathrm{Gr}(k,n)$ (see discussion above Example 1.2.3 in \cite{brion-flags}). Now take the Bott--Samelson variety $Z_{\underline{w}}$ associated to a reduced decomposition $\underline{w}$ of $w$. By e.g. \cite[Proposition 2.2.1 (iii), (iv)]{brion-flags}, this is an iterated fibration of K\"ahler manifolds \begin{align} M_k \to M_{k-1} \to \cdots \to M_1=\mathbb{CP}^1 \end{align} with fiber $\mathbb{CP}^1$, which maps birationally onto $Z_w$ and hence onto $\mathrm{Gr}(k,n)$. The statement again follows from Lemma \ref{lem:blowups1} and Lemma \ref{lem:tower1}.
\end{proof}

\subsection{Proof of Theorems}

We now turn to giving our main arguments. As in previous sections, $\mathbb{E}^*$ will denote a complex-oriented cohomology theory. 
We will make use of the following technical variant of the Leray--Hirsch theorem:

\begin{lemma} \label{lem:lerayhirsch} 
Suppose that we have a c-split fibration $M \xrightarrow{i} P \xrightarrow{\pi} B$, where $B$ is a finite type CW-complex with only even dimensional cells. Then the following hold: \begin{enumerate} \item If $B$ is finite, the natural map \begin{align} \phi: \mathbb{E}^*(M)\otimes_{\mathbb{E}_*} \mathbb{E}^*(B) \to \mathbb{E}^*(P) \end{align} is an isomorphism. \item For general $B$ as above, the c-splitting induces an isomorphism \begin{align} \mathbb{E}^*(M \times B)\cong \mathbb{E}^*(P). \end{align} of $\mathbb{E}^*(B)$-modules.  \end{enumerate} \end{lemma}
\begin{proof}\emph{Part (1):}  We follow the proof of the Leray--Hirsch theorem in \cite[Theorem 4D.1]{Hatcher}.   We assume that $B$ has dimension $m$ and proceed by induction on $m$. We let $B'$ be the subspace obtained by deleting a point $x_\alpha$ from each $m$ cell $U_\alpha$ and let $P'$ denote the restriction of the fibration to $B'$. Then $M \to P' \to B'$ is also a c-split fibration. We let $V_\alpha \subset U_\alpha$ be small disc about $x_\alpha$ over which the fibration is trivial.

We note that the cohomologies $\mathbb{E}^*(B,B')$, $\mathbb{E}^*(B)$, $\mathbb{E}^*(B')$ are all free over $\mathbb{E}_*$. For $\mathbb{E}^*(B)$, $\mathbb{E}^*(B')$ this follows from Lemma \ref{lem:AtiyahHirzebruch}, together with the fact that $B'$ retracts onto the $m-2$ skeleton. For $\mathbb{E}^*(B,B')$, this follows form the excision isomorphism $\mathbb{E}^*(B,B') \cong \mathbb{E}^*(V_\alpha,V_\alpha-x_\alpha).$ Lemma \ref{lem:AtiyahHirzebruch} implies that the maps $\mathbb{E}^*(B) \to \mathbb{E}^*(B')$ are surjections. It follows from the long exact sequence for the pair $(B,B')$ that the connecting maps \begin{align} \label{eq:connectingzero} \delta :\mathbb{E}^*(B')  \to \mathbb{E}^{*+1}(B,B') \text{ are zero.} \end{align} 
This means that the long exact sequence gives rise to a short exact sequence: \begin{align} 0 \to \mathbb{E}^*(B,B') \to \mathbb{E}^*(B) \to  \mathbb{E}^*(B') \to 0 \end{align} 
Because all of these groups are finite free over $\mathbb{E}_*$, we learn that  \begin{align} \label{eq:exactsequence} 0 \to \mathbb{E}^*(B,B')\otimes_{\mathbb{E}_{*}}\mathbb{E}^*(M) \to \mathbb{E}^*(B)\otimes_{\mathbb{E}_{*}}\mathbb{E}^*(M)  \to  \mathbb{E}^*(B')\otimes_{\mathbb{E}_{*}}\mathbb{E}^*(M)  \to 0 \end{align}  is also exact. 
Let $(\mathbb{E}^*(B,B')\otimes_{\mathbb{E}_{*}}\mathbb{E}^*(M))^{q}$ denote the degree $q$ piece of the tensor product and similarly for the other groups. Then there is a commutative diagram:
\[
\xymatrix{
\cdots \ar[r] & (\mathbb{E}^*(B,B')\otimes_{\mathbb{E}_{*}}\mathbb{E}^*(M))^{q}  \ar[d]^\phi \ar[r] & (\mathbb{E}^*(B)\otimes_{\mathbb{E}_{*}}\mathbb{E}^*(M))^{q}  \ar[d]^\phi \ar[r] & (\mathbb{E}^*(B')\otimes_{\mathbb{E}_{*}}\mathbb{E}^*(M))^{q}  \ar[r] \ar[d]^\phi  &\cdots  \\
\cdots \ar[r] &\mathbb{E}^q(P,P') \ar[r] & \mathbb{E}^q(P)\ar[r] & \mathbb{E}^q(P') \ar[r] & \cdots} \]
The top row is exact by \eqref{eq:exactsequence} and \eqref{eq:connectingzero} and the commutativity of the diagram follows as in the proof of \cite[Theorem 4D.1]{Hatcher}. It is then easy to establish that $\phi: \mathbb{E}^*(B,B')\otimes_{\mathbb{E}_{*}}\mathbb{E}^*(M) \to \mathbb{E}^*(P,P')$ is an isomorphism by retracting everything on to $V_\alpha.$ Because $B'$ retracts onto the $m-2$ skeleton, the induction hypothesis implies that the map $\phi: \mathbb{E}^*(B')\otimes_{\mathbb{E}_{*}}\mathbb{E}^*(M) \to \mathbb{E}^*(P')$ is also an isomorphism. The claim now follows from the five-lemma.  \vskip 5 pt

\emph{Part (2):} We let $B^m$ denote the $m$-skeleton of $B$ and let $P^m \to B^m$ denote the restriction of $P$ to $B^m$. Using the splitting $s$, we can identify \begin{align} \phi^m: \mathbb{E}^*(M)\otimes_{\mathbb{E}_*} \mathbb{E}^*(B^m) \to \mathbb{E}^*(P^m)  \end{align} In particular, all of the restriction maps $\mathbb{E}^*(P^{m+1}) \to \mathbb{E}^*(P^m)$ are surjective. Applying the Milnor exact sequence together with the Mittag-Leffler property, we obtain isomorphisms: 
\begin{align} \varprojlim_m \mathbb{E}^*(M)\otimes_{\mathbb{E}_*} \mathbb{E}^*(B^m) \cong \varprojlim_m \mathbb{E}^*(P^m) \cong \mathbb{E}^*(P). \end{align}
Finally, we note that the same argument shows that the inverse system $\varprojlim_m \mathbb{E}^*(M)\otimes_{\mathbb{E}_*} \mathbb{E}^*(B^m)$ is canonically isomorphic to  $\mathbb{E}^*(M \times B)$, proving the claim. 

 \end{proof}

\begin{rem} One can also use standard spectral sequence arguments to give a proof of Lemma \ref{lem:lerayhirsch}. \end{rem}

\begin{cor} \label{cor:productbundle}  Suppose that we have a Hamiltonian fibration $M \xrightarrow{i} P \xrightarrow{\pi} B$, where $B$ is a product of Grassmannians. Then there is a non-canonical isomorphism of $\mathbb{E}^*(B)$ modules:  \begin{align} \mathbb{E}^*(M)\otimes_{\mathbb{E}_*} \mathbb{E}^*(B) \to \mathbb{E}^*(P) \end{align}  \end{cor}
\begin{proof} This follows immediately by combining Lemma \ref{lem:cohsplitproj} and Lemma \ref{lem:lerayhirsch}.  \end{proof}

\begin{lemma} \label{lem:splitting}
Suppose that we have a c-split fibration $M \xrightarrow{i} P \xrightarrow{\pi} B$, where $B$ is a finite type CW-complex with only even dimensional cells. Let $i:B_0 \subset B$ be a sub-complex, $P_0 \to B_0$ denote the restriction of the fibration to $B_0$, and $\iota: P_0 \to P$ the inclusion.  
    Given a fixed c-splitting $s_{0}: \mathbb{E}^*(M) \to \mathbb{E}^*(P_{0})$, we can find a c-splitting $s: \mathbb{E}^*(M) \to \mathbb{E}^*(P)$  such that
    \begin{equation}\label{eq:splittings}
        s_0 = \iota^* \circ s.
    \end{equation}
\end{lemma}

\begin{proof}
Let $\tilde{s}$ be the c-splitting over $B$. This gives rise to a commutative diagram:
 \[
\xymatrix{
 \phi:\mathbb{E}^*(M)\otimes_{\mathbb{E}_*}  \mathbb{E}^*(B)\ar[r] \ar[d]^{\operatorname{id} \otimes i^*} & \mathbb{E}^*(P)\ar[d]^{\iota^*}   \\ 
   \phi: \mathbb{E}^*(M)\otimes_{\mathbb{E}_*}  \mathbb{E}^*(B_0) \ar[r] & \mathbb{E}^*(P_0)}  \]  
where the horizontal arrows are isomorphisms. Thus, because $\mathbb{E}^*(B)$, $\mathbb{E}^*(B_0)$ are free, we may choose a splitting of $\mathbb{E}_*$ modules $\mathbb{E}^*(B_0)\to  \mathbb{E}^*(B)$ for $i^*$. This defines a splitting $\tau: \mathbb{E}^*(P_{0})\hookrightarrow \mathbb{E}^*(P)$ of $\mathbb{E}_*$- modules for $\iota^*.$ We can thus split the cohomology $\mathbb{E}^*(P)$ as \begin{align}\mathbb{E}^*(P) \cong \tau(\mathbb{E}^*(P_{0}))\oplus \operatorname{ker}(\iota^*)\end{align} 

We then define a new splitting $s:\mathbb{E}^*(M) \to \mathbb{E}^*(P)$ using $s_{0}$ by letting $s=(\tau\circ s_0,0)$ in this direct sum decomposition. By construction $s$ satisfies \eqref{eq:splittings}.

\end{proof}

\emph{Proof of Theorem 1.2:} \vskip 5 pt 
  \emph{Part (1):} Recall that there is a convenient model for the total space of the universal $U(k)$-bundle, $EU(k)$, given by taking the set of orthonormal $k$-frames in a complex Hilbert space $\mathcal{H}.$ Topologically, this is the colimit of spaces of $k$-frames in $\mathbb{C}^n$  as $n \to \infty$:  $$ EU(k)_1 \subset \cdots \subset EU(k)_n \subset \cdots .$$ The group $U(k)$ acts freely on this space and the quotient is the infinite Grassmannian of \(k\)-planes $Gr_k(\mathcal{H}).$ Given a product of unitary groups $G=\prod U(k_i)$, we can similarly take $EG:=\prod EU(k_i).$ The classifying space for $G$, $BG$, therefore admits finite dimensional approximations by spaces $BG_n$ which are products of Grassmannians.  We set $M_G:=M\times_G EG$ and $M_{G,n}:=M\times_G EG_n$. In view of Lemma \ref{lem:cohsplit}, we have c-splittings $s_n: \mathbb{E}^*(M) \to \mathbb{E}^*(M_{G,n})$ which by Corollary \ref{cor:productbundle} give rise to identifications: \begin{align} \label{eq:EGn} \mathbb{E}^*(M_{G,n}) \cong \mathbb{E}^*(M)\otimes \mathbb{E}^*( BG_n).\end{align} The identifications from \eqref{eq:EGn} imply that the restrictions  $\mathbb{E}^*(M_{G,n+1}) \to  \mathbb{E}^*(M_{G,n})$ are surjective and we have that \begin{align}\mathbb{E}^*(M_G) \cong \varprojlim_n\mathbb{E}^*(M_{G,n}). \end{align} By Lemma \ref{lem:splitting}, we can arrange that the c-splittings $s_n$ are compatible with restriction and thus give rise to a splitting:  \begin{align}s: \mathbb{E}^*(M) \to \mathbb{E}^*(M_G). \end{align}
  Using Lemma \ref{lem:lerayhirsch} (ii) (by construction the model for $BG$ has only even-dimensional cells), we obtain that
\begin{align} \mathbb{E}^*(M_G) \cong \mathbb{E}^*(M \times BG). \end{align}
 It remains to prove that the Leray-spectral sequence $\lbrace E_r^{p,q}(M_{G}) \rbrace$ for the fibration $M_G \to BG$ degenerates. It follows easily from \eqref{eq:EGn} that the Leray spectral sequences $\lbrace E_r^{p,q}(M_{G,n}) \rbrace$ associated to the fibrations $M_{G,n} \to BG_n$ degenerate. The $E_2$ page $E_2^{p,q}(M_{G})$ is an inverse limit of the $E_2$ pages $E_2^{p,q}(M_{G,n})$, which implies that the differential vanishes on $E_2^{p,q}(M_{G})$. The same argument then implies that the differentials on the subsequent pages $E_r^{p,q}(M_{G})$ are zero, which proves degeneration.  
\vskip 5 pt
\emph{Part (2):} 
We let $EG$ be any CW-complex on which $G$ (and hence $T \subset G$) acts freely. We again set $M_T=M\times_T EG$ and $M_G:=M\times_G EG$ and consider the diagram \[
\xymatrix{
 M_T \ar[d] \ar[r]^{r} & M_G \ar[d]  \\
    BT \ar[r] &BG} \] 

We let $\lbrace E_r^{p,q}(M_T) \rbrace$ and $\lbrace E_r^{p,q}(M_G) \rbrace$ be the Leray spectral sequences associated to the fibrations over the respective Borel spaces $BT$ and $BG$.  By functoriality of the Leray spectral sequence, $r$ induces maps $r^*: E_r(M_G) \to E_r(M_T).$ We can identify the $E_2$ pages of the spectral sequence with $E_2(M_G)=\mathbb{E}^*(M)\otimes_\mathbf{k}H^*(BG,\mathbf{k})$ and $E_2(M_T)=\mathbb{E}^*(M)\otimes_\mathbf{k}H^*(BT,\mathbf{k})$ respectively. By hypothesis $r^*$ is injective on the $E_2$ pages and hence there are no differentials on $E_2(M_G).$ The same argument shows there are therefore no differentials on subsequent pages and thus the spectral sequence $E_r^{p,q}(M_G)=>\mathbb{E}^*(M_G)$ degenerates at the second page as required. \qed \vskip 5 pt

\begin{rem} We note that in case of classical cohomology, the argument for Part (1) actually simplifies a little bit. This is because we can choose $BG_n$ large enough so that the restriction map  \begin{align}\label{eq:splittingEGn} j^*:H^d(M_G) \to  H^d(M_{G,n}) \text{ is an isomorphism in all degrees } d \leq \operatorname{dim}(M). \end{align} By Lemma \ref{lem:cohsplit}, we can produce a c-splitting $\bar{s}: H^*(M) \to H^*(M_{G,n})$ for the fibration $M_{G,n} \to BG_n.$ Then by composing with the inverse maps to \eqref{eq:splittingEGn}, we get a c-splitting $s: H^*(M) \to H^*(M_G)$. \end{rem}

We have the following Corollary.
\begin{cor} \label{cor:randomthing} Let $T \cong (S^1)^m$ be a torus and $M$ a symplectic manifold with Hamiltonian $T$-action. There is an isomorphism of $\mathbb{E}^*(BT) \cong \mathbb{E}_*[\![u_1,\cdots,u_m]\!]$ modules:
\begin{align} \mathbb{E}^*_T(M) \cong \mathbb{E}^*(M)[\![u_1,\cdots,u_m]\!]. \end{align} \end{cor}
\begin{proof} By equivariant formality, $\mathbb{E}^*_T(M) \cong \mathbb{E}^*(M\times BT).$ We also have an identification $\mathbb{E}^*(M\times BT) \cong \mathbb{E}^*(M)[\![u_1,\cdots,u_m]\!]$ from which the result follows.  \end{proof}
\vskip 5 pt

\section{Localization and GKM}
\subsection{HKR theories}

We next turn to discussing a particular class of complex-oriented cohomology theories. Throughout the discussion, we fix a prime number $p$. We let 
$\mathbb{Z}_p$ denote the $p$-adic integers and let $\mathbb{Z}_{(p)}=\mathbb{Z}_p \cap \mathbb{Q}.$ Suppose we have  complex-oriented cohomology theory $\mathbb{E}^*$ where $\mathbb{E}_*$ is a local ring with graded maximal ideal $\mathfrak{m}$ such that $p \in \mathfrak{m}$. For such $\mathbb{E}^*$,  the mod $\mathfrak{m}$ reduction, $L_{0}$, of $L_{\mathbb{E}}$ is a formal group over a graded field of characteristic $p > 0$. The height of such a cohomology theory  is the height of the formal group law $L_{0}$, which we denote by $n$.\begin{defn}\label{defn:HKR} We say that  a complex-oriented cohomology theory $\mathbb{E}^*$ is of \emph{HKR-type} if $\mathbb{E}_*$ is a Noetherian, complete, local ring with (graded-) maximal ideal $\mathfrak{m}$ such that $p \in \mathfrak{m}$ and $\mathbb{E}^*$ is of finite height $n<\infty$. \end{defn} Cohomology theories of this kind were considered by Hopkins--Kuhn--Ravenel \cite[\S 5]{HKR}. Many interesting examples of such theories arise from Johnson--Wilson cohomology $E(n)^*$, which is a cohomology theory with coefficients: \begin{align} E(n)_*= \mathbb{Z}_{(p)}[v_1,\cdots, v_{n-1},v_n,v_n^{-1}],\end{align} where the periodic variable $v_n$ plays a special role and has degree $|v_n|= 2(p^n-1)$. From $E(n)^*$, one can construct the following examples of HKR-type cohomology theories: :\begin{enumerate} \item The $I_n$-adically complete version of $E(n)^*$, $\hat{E}(n)^*$, (\cite[\S 1.3]{HKR}) is a cohomology theory with coefficients the completion of $E(n)^*$ at the ideal $$I_n=(p,v_1,\cdots,v_{n-1}).$$ \item The integral lifts of Morava $K$-theory $\hat{K}(n)$ arise by taking the $p$-completion of the quotient $E(n)/(v_1,\cdots,v_{n-1}).$ See \cite[\S 1.3]{Greenlees} for more details. These theories have coefficients given by:  \begin{align} \hat{K}(n)_* = \mathbb{Z}_p[v_n,v_n^{-1}]. \end{align} \item Finally the Morava $K$-theories $K(n)$ are the cohomology theories which arise by further quotienting by $p$. These cohomology theories have coefficients given by \begin{align} K(n)_{*} = \mathbb{F}_p[v_n,v_n^{-1}]. \end{align} \end{enumerate}

The key to the localization result is the following computation of the $\mathbb{E}^*$-theory for classifying spaces of finite abelian groups from \cite[\S 5.2]{HKR}.

\begin{lemma} \label{lem:MoravaBA} Let $\mathbb{E}^*$ be an HKR theory of height $n<\infty$. Let $A$ be a finite cyclic group of order $\ell$. Write $\ell=p^rs$ with $(s,p^r)=1$. Then \begin{enumerate} \item $\mathbb{E}^*(BA)$ is free over $\mathbb{E}_*$ of rank $p^{rn}$.  \item In the case $\mathbb{E}^*=K(n)$ is the Morava $K$-theory, we have: \begin{align} \label{eq:moravanilpotent} \mathbb{E}^*(BA) \cong \mathbb{E}_*[[u]]/(u^{p^{rn}}). \end{align} \end{enumerate}  \end{lemma}
\begin{proof} Part (1) is established in \cite[Cor. 5.8]{HKR}. To establish Part (2), write $\ell=p^rs$ with $(s,p^r)=1$. Note that by by Part (1), we have that the pull-back map \begin{align} \mathbb{E}^*(BA) \to \mathbb{E}^*(B\mathbb{Z}/p^r\mathbb{Z})  \end{align} is an isomorphism. Thus, it suffices to consider the case where $\ell=p^r.$ View $A$ as a finite subgroup $A \subset S^1$. Regard $\mathbb{E}^*(BA)$ as a module over $\mathbb{E}_*[\![u]\!]$ via the map $BA \to BS^1$. By \cite[Lemma 5.7]{HKR} we have that \begin{align} \label{eq:EBA1} \mathbb{E}^*(BA) \cong \mathbb{E}_*[\![u]\!]/([p^r]\cdot_{\mathbb{E}}u). \end{align} The formal group law $L_\mathbb{E}$ for Morava $K$-theory is essentially characterized by the fact that its $p$-series satisfies: \begin{align} \label{eq:Hondagrouplaw} [p]\cdot_{\mathbb{E}}u = v_n u^{p^{n}}.\end{align}  Iterating the formal group law, we obtain \begin{align} [p^r] \cdot_\mathbb{E} u =v_n^{d(r)} u^{p^{rn}}. \end{align}
for some integer $d(r)>0$. The result follows immediately from this. 
 \end{proof} 

\subsection{Localization}

\begin{defn} \label{defn:HRKfinite} An HKR theory $\mathbb{E}^*$ is \emph{integral of finite Krull dimension} if the coefficient ring  $\mathbb{E}_*$ is an integral domain of finite Krull dimension. \end{defn}

\begin{rem} All of the examples $\hat{E}(n)^*,\hat{K}(n)^*, K(n)^* $ are integral of finite Krull dimension. \end{rem}
 
Fix an HKR cohomology theory which integral and of finite Krull dimension. The goal of this section is to extend the Atiyah--Bott localization theorem (\cite[Theorem 3.5]{Atiyah-Bott}) to $\mathbb{E}^*$ coefficients.  Throughout this section we take our group to be a torus $T \cong (S^1)^m$. We have by \cite[\href{https://stacks.math.columbia.edu/tag/0BNH}{Tag 0BNH}]{stacks-project} that $\mathbb{E}^*(BT)$ is Noetherian\footnote{Strictly speaking \cite[\href{https://stacks.math.columbia.edu/tag/0BNH}{Tag 0BNH}]{stacks-project} considers standard completions as opposed to graded completions. However the result and proof apply without change to graded completions. } and it is also an integral domain.  We may therefore set $R=\operatorname{Frac}(\mathbb{E}^*(BT))$ to be the fraction field of $\mathbb{E}^*(BT).$ We have the following version of the Atiyah-Bott support lemma (cf. \cite[Lemma 3.4]{Atiyah-Bott}): 

\begin{lemma} \label{lem:ABsupport} Let $X=T/K$ be an orbit of $T$, where $K$ is a closed subgroup of $T$. Let $\operatorname{supp}(\mathbb{E}^*_T(X)) \subset \operatorname{Spec}(\mathbb{E}^*(BT))$ denote the support of $\mathbb{E}^*_T(X)$ as a module over $\mathbb{E}^*(BT)).$ Then: \begin{enumerate} \item $\operatorname{supp}(\mathbb{E}^*_T(X))$ is a closed subspace of codimension $\operatorname{dim}(T)-\operatorname{dim}(K)$. \item In the case where $\mathbb{E}^*=K(n)^*$ is Morava K-theory, $\operatorname{supp}(\mathbb{E}^*_T(T/K))=\operatorname{supp}(\mathbb{E}^*_T(T/K^0))$, where $K^0 \subset K$ is the connected component of $K$ at the identity.  \end{enumerate}  \end{lemma}
\begin{proof}   We have that $\mathbb{E}_{T}^*(X) \cong \mathbb{E}^*(BK)$ and the induced $\mathbb{E}^*(BT)$ module structure on this group is induced by the inclusion $K \subset T$. We first argue that the map \begin{align}\mathbb{E}^*(BT) \to \mathbb{E}^*(BK) \end{align} is surjective so that $\operatorname{Spec}(\mathbb{E}^*(BK))$ is a closed subscheme of $\operatorname{Spec}(\mathbb{E}^*(BT))$.  As above, we let $K^0$ denote the connected component of $K$. After choosing a complementary torus $H$ to $K^0$, we may write $K=A \times K^0$, where $A$ is a finite abelian subgroup of $H$. The by the Kunneth isomorphism \cite[Cor. 5.11]{HKR}, we have \begin{align} \label{eq:Kunneth} \mathbb{E}^*(BA)\otimes \mathbb{E}^*(BK^0) \cong \mathbb{E}^*(BK). \end{align}  Moreover, the Kunneth map fits into a diagram: 
\[
\xymatrix{
   \mathbb{E}^*(BH)\otimes \mathbb{E}^*(BK^0) \ar[d] \ar[r] & \mathbb{E}^*(BT) \ar[d]  \\
    \mathbb{E}^*(BA)\otimes \mathbb{E}^*(BK^0)\ar[r] &\mathbb{E}^*(BK)} \] 
It therefore suffices to assume that $K=A$ is a finite abelian group. Moreover, by Part (1) of Lemma \ref{lem:MoravaBA} we can assume that $A \cong \mathbb{Z}/p^{r_1} \times \mathbb{Z}/p^{r_2} \times \cdots \times \mathbb{Z}/p^{r_i}$ is a product of $p$-groups. It is easy to see that we may find:\begin{itemize} \item a rank one subgroup $H' \cong S^1 \subset T$ so that $\mathbb{Z}/p^{r_1} \subset H'$ \item a complementary torus $H''$ (for $H'$)  together with a subgroup $A' \subset H''$ so that $A \cong \mathbb{Z}/p^{r_1} \times A'$ (the group $A' \cong \mathbb{Z}/p^{r_2} \times \cdots \times \mathbb{Z}/p^{r_i}$). \end{itemize} The surjectivity of the restriction $\mathbb{E}^*(BT) \to \mathbb{E}^*(BA)$ follows from the Kunneth isomorphism together with induction on the number of factors in the product decomposition of $A$. To prove the claim on dimensions, we note that the claim is obvious for connected groups where $K=K^0$. In view of the Kunneth isomorphism and Part (1) of Lemma \ref{lem:MoravaBA}, we have that $\mathbb{E}^*(BK)$ is finite over $\mathbb{E}^*(BK^0)$, which implies that they have equal dimensions \cite[\href{https://stacks.math.columbia.edu/tag/01WG}{Tag 01WG}]{stacks-project}. To prove Part (2) of the Lemma, we again write our group $K$ as $A\times K^0$. Without loss of generality, we can assume $A \cong \mathbb{Z}/p^{r_1} \times \mathbb{Z}/p^{r_2} \times \cdots \times \mathbb{Z}/p^{r_i}$. Then Part (2) of Lemma \ref{lem:MoravaBA} shows that $K(n)^*(BA)$ is a tensor product of the algebras in \eqref{eq:moravanilpotent}. In particular, $K(n)^*(BA)$ is a nilpotent extension of $K(n)_*$ for any finite abelian group. The claim now follows from the Kunneth isomorphism \eqref{eq:Kunneth}.\end{proof}



For the rest of this section, unless otherwise specified, let $M$ be a smooth compact manifold with $T$-action and denote by $M^{T}$ the set of fixed points of the $T$-action.

\begin{lemma} \label{lem:localization} Let $j: M^T \to M$ denote the inclusion of $M^T$ into $M$. Then the maps: \begin{align} j^*:\mathbb{E}_{T}^*(M)\otimes_{\mathbb{E}^*(BT)} R \to \mathbb{E}_{T}^*(M^T)\otimes_{\mathbb{E}^*(BT)} R \end{align} is an isomorphism. \end{lemma}
\begin{proof}  We stratify the complement of an equivariant tubular neighborhood $U$ of $M^T$ by finitely many open sets which equivariantly retract onto non-constant orbits. Note that for a non-constant orbit $X$, Lemma \ref{lem:ABsupport} implies in particular that \begin{align} \mathbb{E}_{T}^*(X)\otimes_{\mathbb{E}^*(BT)} R =0. \end{align}   The result then follows from Lemma \ref{lem:MoravaBA}, Lemma \ref{lem:ABsupport}, and the Mayer-Vietoris sequence.    \end{proof}

\begin{proof}[Proof of Theorem \ref{thm:highlight}] The assumption that $\mathbb{E}_*$ is free implies by Corollary \ref{cor:randomthing} that \begin{align} \mathbb{E}_T^*(M) \cong \mathbb{E}^*(M)[[u_1,\cdots,u_m]] \cong \mathbb{E}^*(M)\otimes_{\mathbb{E}_*} \mathbb{E}^*(BT). \end{align} Hence, $\mathbb{E}_T^*(M)$ is free as a module over $\mathbb{E}^*(BT).$ By Lemma \ref{lem:localization}, the kernel of $j^*$ is a torsion module over $\mathbb{E}^*_T(pt)$. Because $\mathbb{E}^*_{T}(M)$ is free, it follows that the kernel of $j^*$ is trivial. \end{proof}

To end this subsection, we include an Atiyah--Bott localization formula for HKR theories.

\begin{thm}\label{AB-localization}
Suppose that $M$ has an almost complex structure which is preserved under the $T$-action. Index the connected components of $M^T$ by the set $I$, and for each $i \in I$, write the normal bundle of the corresponding connected component $\nu_i$. Then the pushforward on the localized cohomology 
\begin{equation}
    j_{*}:\mathbb{E}_{T}^*(M^T)\otimes_{\mathbb{E}^*(BT)} R \to \mathbb{E}_{T}^*(M)\otimes_{\mathbb{E}^*(BT)} R
\end{equation}
is an isomorphism with inverse
\begin{equation}
    \sum_{i \in I} \frac{1}{\mathit{eu}^{T}_{\mathbb{E}}}(\nu_i) j_{i}^*.
\end{equation}
\end{thm}
\begin{proof} 
By the complex-orientation assumption and its compatibility with the $T$-action, we see that the pushforward in equivariant cohomology $j_*$ and the Euler class of the normal bundle $\nu_i$ are well-defined. Because for each $i \in I$, we have $j_i^* j_* (1) = \mathit{eu}^{T}_{\mathbb{E}}(\nu_i)$. Using the Mayer--Vietoris long exact sequence for the pair $(M, M \setminus M^T)$, as in Lemma \ref{AB-localization}, one can show that for the pushforward $j_*$, both $\ker(j_*)$ and $\mathrm{coker}(j_*)$ are both torsion modules over $\mathbb{E}^*(BT)$. Combining with Lemma \ref{AB-localization}, it follows that $j^* j_*$ is an isomorphism. This proves the invertibility of the Euler classes, so the statement follows accordingly.
\end{proof}

\begin{rem}
     Following \cite[Section 3]{Atiyah-Bott}, we can prove Theorem \ref{AB-localization} in a slightly different way. For the connected component $M^T_i$, we can decompose its normal bundle into a direct sum of irreducible $T$-eigenbundles
\begin{equation}
    \nu_i = \nu_i^{(1)} \oplus \cdots \oplus \nu_i^{(k)}
\end{equation}
such that for each $1 \leq \alpha \leq k$, the corresponding character $\varphi_{\alpha}: T \to U(1)$ is nontrivial. If we write $u_{\alpha} \in \mathbb{E}_{*}[\![ u_1, \dots, u_m ]\!]$ the image of the generator of $\mathbb{E}^*(BS^1)$ in $\mathbb{E}^*(BT)$ induced by $\varphi_{\alpha}$, then 
\begin{equation}
   \mathit{eu}^{T}_{\mathbb{E}}(\nu_i) = \prod_{1 \leq \alpha \leq k} (\mathit{eu}_{\mathbb{E}}(\nu_i^{(\alpha)}) +_{\mathbb{E}} u_{\alpha}),
\end{equation}
where we abuse the notation using $+_{\mathbb{E}}$ to denote the element in $\mathbb{E}_{T}^*(M^T)$ obtained by plugging in the given two elements in the formal group law associated with $\mathbb{E}^*$. The Euler class of line bundles, which coincides with the first Chern class in $\mathbb{E}^*$, is nilpotent. Indeed, for any finite-dimensional CW complex $X$, the set of isomorphism classes of complex bundles on $X$ have a one-to-one correspondence with the set of homotopy classes of maps from $X$ to $\mathbb{CP}^{\infty}$. If a line bundle $L \to X$ is classified by $f: X \to \mathbb{CP}^{\infty}$, then $\mathit{eu}_{\mathbb{E}}(L) = f^*(u)$, where $u$ is the generator of $\mathbb{E}^*(\mathbb{CP}^{\infty})$. By naturality, we know that $\mathit{eu}_{\mathbb{E}}(L)^N = f^*(u^N)$ for any $N \in \mathbb{Z}$. On the other hand, $u^N$ is the Poincar\'e dual of a complex codimension $N$ subspace in $\mathbb{CP}^{\infty}$. So, for $N \gg 1$, we can choose a homotopy of $f$ to show that $f^*(u^N) = 0$, which proves the nilpotency. Accordingly, we see that $\mathit{eu}^{T}_{\mathbb{E}}(\nu_i)$ is invertible after tensor with the fraction field $R$.
\end{rem}

\subsection{GKM theorem}
In this subsection, we specialize our discussion to the case where $\mathbb{E}^*$ is a Morava $K$-theory $K(n)$. We show that the so-called Chang--Skjelbred lemma works verbatim for $K(n)$ coefficients, and deduce the GKM calculation of equivariant cohomology as a corollary.

Let $M$ be a compact smooth $T$-manifold. Denote by $\mathcal{P}$ the set of closed subschemes $\mathcal{Z}_K:=\operatorname{supp}(\mathbb{E}^*_T(T/K)) \subset \mathrm{Spec}(\mathbb{E}_*[\![ u_1, \dots, u_m ]\!])$ where $K$ is a stabilizer of some point $x \in M\setminus M^T$ (we give $\mathcal{Z}_K$ the reduced-induced scheme structure). This is a partially ordered set, with partial order given by inclusion. Note that by Lemma \ref{lem:ABsupport} (2), the scheme $\mathcal{Z}_K$ only depends on the connected component of the identity element $K^0$.

Let $\delta'$ denote the connecting homomorphism of the $\mathbb{E}^*_{T}$-cohomology of the pair $(M, M^T)$. Given $\xi \in \mathbb{E}^*_{T}(M^T)$, write $I(\xi)$ the ideal of annihilators of of $\delta'(\xi)$. On the other hand, for each $\mathfrak{m} \in \mathcal{P}$, we write $M^{\mathfrak{m}}$ the subset of $M$ consisting of points $x \in M$ whose stabilizer group $K_x$ satisfies $\mathfrak{m} \subset \mathcal{Z}_{K_x}$. Denote the connecting homomorphism of $\mathbb{E}^*_{T}$-cohomology of the pair $(M, M^{\mathfrak{m}})$ by $\delta^{\mathfrak{m}}$. The following technical statement is proved in the same way as \cite[Lemma 2.3]{CS} and \cite[Lemma 15.8]{GKM}.

\begin{lemma}\label{lemma:CS}
    For $\xi \in \mathbb{E}^*_{T}(M^T)$, if $\delta'(\xi) \neq 0$, then the variety defined by $I(\xi)$ satisfies
    \begin{equation}
    V(I(\xi)) \subset \bigcup_{\substack{\mathfrak{m} \in \mathcal{P} \\ \delta^{\mathfrak{m}}(\xi) \neq 0}} \mathfrak{m}.
    \end{equation}
    in the scheme $\mathrm{Spec}(\mathbb{E}_*[\![ u_1, \dots, u_m ]\!])$. \qed
\end{lemma}

Denote by $M_1 \subset M$ the set of points consisting of the union of $0$ and $1$-dimensional orbits of $T$. In particular, $M^{T} \subset M_1$. 

\begin{prop}\label{prop:6.3}
    Let $\mathbb{E}^*$ be a Morava $K$-theory. If the equivariant cohomology $\mathbb{E}^*_T(M)$ is a free module over $\mathbb{E}^*(BT)$, then the sequence
    \begin{equation}\label{eqn:exact-sequence}
        0 \longrightarrow \mathbb{E}^*_T(M) \xrightarrow{j^*} \mathbb{E}^*_{T}(M^T) \xrightarrow{\delta} \mathbb{E}^{*+1}_T(M_1, M^T)
    \end{equation}
    is exact, where $\mathbb{E}^{*+1}_T(M_1, M^T)$ denotes the relative equivariant cohomology of the pair $(M_1, M^T)$ and $\delta$ is the connecting homomorphism.
\end{prop}
\begin{proof}
    The fact that $\ker(j)^*$ is trivial follows immediately from Theorem \ref{thm:highlight} because the coefficient ring is a graded field. Because the connecting homomorphism $\delta$ factors as the composition
    \begin{equation}
        \mathbb{E}^*_{T}(M^T) \xrightarrow{\delta'} \mathbb{E}^{*+1}_T(M, M^T) \rightarrow \mathbb{E}^{*+1}_T(M_1, M^T)
    \end{equation}
    where recall that $\delta'$ is the connecting homomorphism for the pair $(M,M^T)$, the long exact sequence implies that $\delta \circ j^* = 0$. To finish the proof, suppose $\xi \in \mathbb{E}^*_{T}(M^T)$ satisfies $\delta(\xi) = 0$. To show it lies in the image of $j^*$, we need to show that $\delta'(\xi) = 0$.
    
    Suppose this is not the case. We claim that the ideal $I(\xi)$ is a principal. We follow the proof of \cite[Lemma 2.2]{CS}. Indeed, let $0 \neq a \in I(\xi)$ be an element which is not divisible by any other element and suppose $b \in I(\xi)$. Choosing a basis $\{x_i\} \subset \mathbb{E}^*_T(M)$, we can find $a_i, b_i \in \mathbb{E}^*_{T}(pt)$ such that $a\xi = j^*(\sum a_i x_i)$ and $b \xi = j^*(\sum b_i x_i)$ using the long exact sequence for the pair $(M, M^T)$. Then we know that $j^*(\sum b a_i x_i) = j^*(\sum a b_i x_i)$, so $\sum b a_i x_i = \sum a b_i x_i$ by the injectivity of $j^*$. This shows that for each $i$, we have $b a_i = a b_i$. As the $a_i$ must have no common factor, we see that $b$ is divisible by $a$, which proves the claim.

    By Lemma \ref{lemma:CS}, as $I(\xi)$ is principal, there must be one $\mathfrak{m}$ which has codimension $1$ in $T$. Because of the inclusion $M^{\mathfrak{m}} \subset M_1$, the connecting homomorphism $\delta^K$ factors as
    \begin{equation}
        \mathbb{E}^*_{T}(M^T) \xrightarrow{\delta} \mathbb{E}^{*+1}_T(M_1, M^T) \to \mathbb{E}^{*+1}_T(M^{\mathfrak{m}}, M^T),
    \end{equation}
    which tells us that $\delta(\xi) \neq 0$, leading to a contradiction.
\end{proof}

\begin{proof}[Proof of Theorem \ref{thm:GKM}:]
This follows from the proof of \cite[Theorem 7.2]{GKM}, which we reproduce for completeness. Because $\mathbb{E}^*_T(M)$ is a free module over $\mathbb{E}^*(BT)$, using the exact sequence \eqref{eqn:exact-sequence}, it suffices to determine the kernel of $\delta: \mathbb{E}^*_T(M^T) \to \mathbb{E}^{*+1}_T(M_1, M^T)$. Given the complement of poles $E_j$, by definition, its stabilizer group is $\ker(\Theta_j)$. Therefore, 
\begin{equation}
\mathbb{E}^*_{T}(E_j) \cong \mathbb{E}^*_{T}(T / \ker(\Theta_j)) \cong \mathbb{E}^*_{\ker(\Theta_j)}(pt).
\end{equation}
Choosing a $T$-invariant open cover of $\ov{E}_j = U_j \cup V_j$ where $U_j = \ov{E}_j \setminus \{ x_{j_0} \}$ and $V_j = \ov{E}_j \setminus \{ x_{j_{\infty}} \}$. Then the corresponding Mayer--Vietoris long exact sequence splits into short exact sequences because all the cohomology groups concentrate in even degrees. Furthermore, we have a commutative diagram
\begin{equation}
     \begin{tikzcd}
0 \arrow[r] & \mathbb{E}^*_{T}(\ov{E}_j)  \arrow[r] \arrow[d, "\cong"] & \mathbb{E}^*_{T}(U_j) \oplus \mathbb{E}^*_{T}(V_j) \arrow[r] \arrow[d, "\cong"]       & \mathbb{E}^*_{T}(E_j) \arrow[r] \arrow[d, "\cong"]                       & 0 \\
0 \arrow[r] & \mathbb{E}^*_{T}(\ov{E}_j)  \arrow[r]                    & \mathbb{E}^*_{T}(x_{j_{\infty}}) \oplus \mathbb{E}^*_{T}(x_{j_0}) \arrow[r, "\delta"] & {\mathbb{E}_{T}^{*+1}(\ov{E}_j, \{x_{j_0}, x_{j_{\infty}} \})} \arrow[r] & 0,
\end{tikzcd}
\end{equation}
which tells us that $\delta$ agrees with the map
\begin{equation}
    (f_{j_0}, f_{j_{\infty}}) \mapsto f_{j_0} |_{\mathbb{E}^*_{\ker(\Theta_j)}} - f_{j_{\infty}} |_{\mathbb{E}^*_{\ker(\Theta_j)}}.
\end{equation}
Ranging over all $j$, the theorem is proved.
\end{proof}

\begin{rem}
    In view of the calculation Lemma \ref{lem:MoravaBA} (2), one can see that for an invariant $2$-sphere with orbit type $T/K$ such that in the decomposition $K = A \times K^0$, the finite group component $A$ is has order greater than $1$, the formula in Theorem \ref{thm:GKM} depends on the height for the cohomology theory $\mathbb{E}^*$. We thank Ivan Smith for a related question.
\end{rem}

\bibliography{ref}

\bibliographystyle{amsalpha}

\end{document}